\documentclass{amsart}
\usepackage{}
\usepackage{color}
\usepackage{amssymb}

\newcommand{\tr}{\mbox{tr}}

\newcommand{\ol}{\overline}
\newcommand{\ul}{\underline}

\newtheorem{theorem}{Theorem}[section]
\newtheorem{lemma}[theorem]{Lemma}
\newtheorem{proposition}[theorem]{Proposition}

 \theoremstyle{definition}
\newtheorem{definition}[theorem]{Definition}

\theoremstyle{remark}

\numberwithin{equation}{section}

\begin{document}

\title[Dirichlet for Monge-Amp\`ere type equations ]
{The Dirichlet problem for Monge-Amp\`ere type equations on Hermitian manifolds}

\author{Weisong Dong}
\address{School of Mathematics, Tianjin University, 135 Yaguan Rd.,
         Tianjin, 300354, P.R. China}
\email{dr.dong@tju.edu.cn}


\begin{abstract}

In this paper, we solve the Dirichlet problem for Monge-Amp\`ere type equations for $(n-1)$-plurisubharmonic functions on Hermitian manifolds.



\end{abstract}

\maketitle

\section{Introduction}

Let $(M, \omega)$ be a compact complex manifold of complex dimension $n \geq 2$ with Hermitian metric $\omega$.
Let $\chi(z)$ be a smooth real (1,1) form on $M$
and $\psi(z) \in C^{\infty}( M )$ be a positive function.
Given any smooth function $u \in C^{\infty}(M)$,
the operator $\mathcal{M}_p^n [u]$ is defined as below
\[
\mathcal{M}_p^n ( u )
= \Pi_{1\leq i_1 < \cdots < i_p\leq n} (\lambda_{i_1} + \cdots + \lambda_{i_p}),
\]
where $\lambda = (\lambda_1, \cdots, \lambda_n )$ are the eigenvalues of $g := \chi + \sqrt{-1} \partial \bar\partial u$ with respect to
$\omega$.
We also simply write $\mathcal{M} (u)$ when there is no ambiguity.
In this paper, we consider the following Monge-Amp\`ere type equation
\begin{equation}\label{eqn}
 \mathcal{M}_p^n ( u ) = \psi \;\mbox{in} \; M,
\end{equation}
with Dirichlet boundary data
\begin{equation}\label{eqn-b}
  u = \varphi \; \mbox{on} \; \partial M,
\end{equation}
where $\varphi \in C^{\infty} (\partial M, \mathbb{R})$ is a given function.

\begin{definition}
\label{def}
Let $p\in \{1, \cdots, n\}$. The cone $\mathcal{P}_p$ is a subset in $\mathbb{R}^n$ with element
$(\lambda_1, \cdots, \lambda_n)$ such that $\lambda_{i_1} + \cdots + \lambda_{i_p} > 0$
for all $1\leq i_1 < i_2 < \cdots < i_p\leq n$.
The cone of $n \times n$ Hermitian matrices $P_p$ is defined as:
$A \in P_p$ if the $n$-tuple of it's eigenvalues is in $\mathcal{P}_p$.
We call $A$ is $p$-positive if $A\in P_p$.
\end{definition}

For any smooth real $(1,1)$-forms $h$ on $(M, \omega)$,
written in local coordinates as $h = \sqrt{-1} h _{i\bar j} d z^i \wedge d z^{\bar j}$,
we say that
\[
h \in P_p (M)
\]
if the vector of eigenvalues of the Hermitian endomorphism ${h^i}_j = \omega^{i\bar k} h_{j \bar k}$ lies in the
$\mathcal{P}_p$ cone at each point of $M$.
A function $u\in C^2(M)$ is called admissible if $\chi + \sqrt{-1} \partial \bar \partial u \in P_p(M)$.
In this paper, we prove the following result.

\begin{theorem}\label{thm1}
Let $(M,\omega)$ be a compact Hermitian manifold with boundary.
Let $\chi \in P_{n-1}(M)$ be a smooth $(1,1)$ form.
Assume that there exists an admissible subsolution $\ul u \in C^\infty (M)$, that is
\[
\mathcal{M}_{n-1}^n (\ul u) \geq \psi \; \mbox{in}\;  M \;\; \mbox{and} \;\; \ul u = \varphi \; \mbox{on}\; \partial M.
\]
Then there exists a unique admissible solution $u \in C^\infty (M)$ to \eqref{eqn} and \eqref{eqn-b} for $p=n-1$.
\end{theorem}

The equation \eqref{eqn} with $p=n-1$ is of special interest, which is closely related to the Gaudochon conjecture,
the form-type Calabi-Yau equation and the
$n-1$ plurisubharmonic functions studied by Harvey-Lawson \cite{HL12,HL13}.
Since Yau in \cite{Yau} proved the existence of a solution to the complex Monge-Amp\`ere equation
on K\"ahler manifolds, that corresponds to the end point $p=1$ in \eqref{eqn},
there have been many extensions and generalizations of Yau's theorem.
Tian-Yau was concerned with the existence of complete Ricci-flat K\"ahler metrics on quasiprojective varieties in \cite{TY1,TY2}.
Gaudochon conjecture \cite{Gauduchon} is a natural generalization of the Calabi-Yau theorem from K\"akler to Hermitian manifolds
and was solved by Sz\'ekelyhidi-Tosatti-Weinkove \cite{S-T-W} recently. See also Guan-Nie \cite{GN}.
The form-type Calaby-Yau equation was first introduced by Fu-Wang-Wu \cite{FWW1} and they solved the
equation with the assumption that the metric has non-negative orthogonal bisectional curvature in \cite{FWW2}.
This assumption was removed in the work of Tosatti-Weinkove \cite{TW17}.
Fino-Li-Salamon-Vezzoni \cite{FLSV} studied the Calabi-Yau equation on certain symplectic non-K\"ahler 4-manifolds,
which is equivalent to solve the Monge-Amp\`ere equation on 2-torus. Recently, Y. Li and the author \cite{DLi} solved the Monge-Amp\`ere
equation on flat $n$-torus involving gradient terms in the determinant.
For the study of Calabi-Yau metric on non-compact singular varieties, we refer the reader to Collins-Guo-Tong \cite{CGT}
and references therein.
The complex Monge-Amp\`ere equation was solved on Hermitian manifolds by Tosatti-Weinkove \cite{T-W2},
and on almost Hermitian manifolds by Chu-Tosatti-Weinkove \cite{C-T-W}.

The Dirichlet problem for complex Monge-Amp\`ere equation in a domain $\Omega \subset \mathbb{C}^n$ was solved by
Caffarelli-Kohn-Nirenberg-Spruck \cite{CKNS} when $\Omega$ is a strongly pseudoconvex domain.
B. Guan explored the problem on general domains with the existence of a subsolution in \cite{Guan98}.
The technique in \cite{Guan98} of deriving second order estimates was used by P. F. Guan in \cite{GuanP}
to resolve the conjecture of Chern-Levine-Nirenberg.
Li in \cite{S.Y.Li} treated the problem for more Hessian equations.
On complex manifolds with boundary, the Dirichlet problem for complex Monge-Amp\`ere equation was widely concerned.
Cheng-Yau \cite{ChY} proved the existence of a solution $u$ to the problem on K\"ahler manifolds with $u = + \infty$ on the boundary.
Cherrier-Hanani \cite{C-H} investigated the problem on strongly pseudoconvex Hermitian manifolds.
Assuming that there exists a subsolution, Guan-Li \cite{G-L} considered the problem on general Hermitian manifolds.
We refer the reader to the survey of Phong-Song-Sturm \cite{P-S-S} for the vast field of complex Monge-Amp\`ere equation.
For a type of Hessian quotient equations, see Guan-Sun \cite{GSun} and Song-Weinkove \cite{SW}.
Weak solutions to complex Hessian equations have been discussed by many authors.
We refer the reader to \cite{BT,BT-acta,Blocki,D-K-weak,Ko,KNguyen1,KNguyen2,LuCH}.
For the real Hessian equation, see \cite{CNS3,CLN,Li,G,GS,JL} and references therein.

The other end point case $p=n$ is also well known since it is a linear PDE.
The two end point case $p=1$ and $p=n$ respectively correspond to the end point case $k=n$ and $k=1$ of the following
counterpart to equation \eqref{eqn}, i.e. the complex $k$-Hessian equations,
\[
\sigma_k (\lambda) = \sum_{1\leq i_1 < \cdots < i_k \leq n} \lambda_{i_1} \cdots \lambda_{i_k} = \psi,\;\mbox{where}\; 1\leq k \leq n.
\]
The above equation arose from the study of Hull-Strominger system by Fu-Yau \cite{F-Y2}.
The complex $k$-Hessian equations for $1<k <n$ on closed K\"ahler manifolds was solved by Dinew-Ko\l odziej \cite{D-K}
in 2012 by first establishing a Liouville theorem, and then combining it with the second order estimate derived by Hou-Ma-Wu \cite{H-M-W}.
Sun in \cite{Sun1} deduced the existence of a solution to the Hessian quotient equation.
The complex $k$-Hessian equations on Hermitian manifolds was solved independently by Zhang \cite{D.K.Zhang} and Sz\'ekelyhidi \cite{Gabor}
by virtue of the Liouville theorem.
The Liouville theorem was employed in a blow-up argument to conclude the missing part of the a priori estimate, the gradient estimate.
Zhang in \cite{X.W.Zhang} proved the gradient estimate with an extra assumption that $\chi + \sqrt{-1} \partial \bar\partial u > 0$.
It is still open to derive the gradient estimate by the maximum principle.
The Dirichlet problem for complex $k$-Hessian equations was solved by Collins-Picard \cite{CP} on Hermitian manifolds recently.

Now we discuss the proof of our main theorem. 
The approach to tackle the equation is the well-known and standard continuity method, 
which reduced the solvability to the a priori estimates.
By the envelope trick in \cite{TWWY}, and the Evans-Krylov theorem,
one can derive the $C^{2,\alpha}$ estimate as long as one have $C^0$, $C^1$ and $C^2$ estimate.
The $C^0$ estimate and the boundary $C^1$ estimate are easy to obtain since we assumed a subsolution.
Therefore, our main goal is to derive the following second order estimate
\[
\sup_M |\sqrt{-1}\partial \bar \partial u | \leq C (1 + \sup_{M} |\nabla u|^2 ),
\]
where $C$ only depends on the background data, but not on $u$. 
This was achieved by following closely the argument in Collins-Picard \cite{CP}.
Then, by virtue of the Liouville theorem for maximal $(n-1)$-plurisubharmonic functions
established by Tosatti-Weinkove \cite{TW17} we can apply a similar blow up argument as in Collins-Picard \cite{CP}
to get the desired gradient estimate, and thereby obtain the second order estimate.

The rest of the paper is organized as follows.
In Section 2, we introduce some useful notations and recall some preliminary estimates.
In Section 3, we prove the boundary mixed normal-tangential estimates.
Finally, we establish the boundary double normal estimate in Section 4.

\textbf{Acknowledgements}:
I would like to thank Professor YanYan Li for helpful comments and for his constant encouragement. 
The author is supported by the National Natural Science Foundation of China, No.11801405.

\section{Preliminaries and main results}

For an admissible solution $u \in C^{\infty}(M)$, in local coordinates,
denote ${g^i}_j = \omega^{i\bar k} (\chi_{ \overline{k} j} + u_{\overline{k} j} )$.
Let $\lambda = (\lambda_1, \cdots, \lambda_n )$ be the eigenvalues of the endomorphism ${g^i}_j$.
Then equation \eqref{eqn} can be rewritten as
\begin{equation}\label{Sk2}
F ( {g^i}_j ) := F (\lambda) := \Pi_{1\leq i_1 < \cdots < i_p\leq n} (\lambda_{i_1} + \cdots + \lambda_{i_p})
= \psi,
\end{equation}
Denote $\tilde F = F^{1/C_n^p}$, and for convenience, we use the notation $f (\lambda) = \tilde F (\lambda)$.
We remark that $f$ satisfies
the following structure conditions:
\begin{align}
 f > 0 \; \mbox{in}\;  \mathcal{P}_p \; \mbox{and}\; f = 0 \; \mbox{on}\; \partial \mathcal{P}_p; \tag{f1}
\end{align}
\begin{align}
 f_i := \frac{\partial f}{\partial \lambda_i} > 0, \; \forall \; 1\leq i\leq n; \tag{f2}
\end{align}
\begin{align}
 f \; \mbox{is concave in} \;\mathcal{P}_p; \tag{f3}
\end{align}
\begin{align}
 f \; \mbox{is homogeneous of degree one, i.e. }\;  f(t \lambda) = t f(\lambda), \; \forall\;  t > 0.\tag{f4}
\end{align}
We refer the reader to \cite{Dinew} for the properties of the operator, or to the appendix in \cite{Dong}.
By (f3) and (f4), we can derive that
\begin{align}
\sum f_i (\lambda) = f (\lambda) + \sum f_i (\lambda) (1 - \lambda_i) \geq f (1,\ldots, 1) > 0 \; \mathrm{ in } \; \mathcal{P}_p. \tag{f5}
\end{align}
Note that $\tilde F (A)$ is concave for $A \in P_p$.
As in Collins-Picard \cite{CP}, we define the tensor
\[
F^{p\bar q} = \frac{\partial F}{\partial {g^r}_p} \omega^{r \bar q}\; \mbox{and}\;
F^{p\bar q, r\bar s} = \frac{\partial^2 F}{\partial {g^a}_p \partial {g^b}_r} \omega^{a\bar q} \omega^{b\bar s}.
\]
At a point $p_0\in M$, where $\omega=\sqrt{-1} \delta_{k \ell} dz^{k} \wedge d \overline{z}^{\ell}$
and ${g^i}_j = \lambda_j {\delta^i}_j$ is diagonal,
we have
\[
\tilde F^{p\bar q} = f_p \delta_{pq}\; \mbox{and} \; \tilde F^{k\bar k} = \frac{\partial f}{\partial \lambda_k} =
\sum_{k\in\{i_1, \cdots, i_p\} } \frac{\frac{1}{C_n^p} f(\lambda) }{\lambda_{i_1} + \cdots + \lambda_{i_p} }.
\]
Now we introduce the following notations that
\[
K = \sup_M |\nabla u|^2 + 1 \;\mbox{and}\; \mathcal{F} = \sum_i f_i.
\]

We recall some known estimates in the literature.
We first recall the maximum principle.
\begin{lemma}[\cite{CNS3}]
\label{MP}
Let $(M, \omega)$ be a compact Hermitian manifold with boundary. Suppose that $v : M \rightarrow \mathbb{R}$ is a smooth function such that
the vector of eigenvalues of $\omega^{-1} (\chi + \sqrt{-1} \partial \bar\partial v )$ lies outside the set
$\mathcal{P}_p^\psi :=\{\lambda \in \mathcal{P}_p: f(\lambda) \geq \tilde \psi\}$ for all $z\in M$. If $u \leq v$ on $\partial M$,
then
\[
u\leq v \; \mbox{on} \; M.
\]
\end{lemma}
\begin{proof}
Suppose $u > v$ at some interior point $x$ in $M$. Without loss of generality, we assume that $u-v$ attains a maximum at $x$.
Then $D^2 u (x) \leq D^2 v(x)$.
Let $\lambda(z) = (\lambda_1, \cdots, \lambda_n)$ be the eigenvalues of $\omega^{-1} (\chi + \sqrt{-1}\partial \ol \partial u) (z)$
and $\mu(z) = (\mu_1, \cdots, \mu_n)$ be the eigenvalues of $\omega^{-1} (\chi + \sqrt{-1}\partial \ol \partial v) (z)$,
both arranged in decreasing order.
By Weyl inequality, we have $\lambda_i \leq \mu_i$ for $1\leq i \leq n$.
Since $\mathcal{P}_p + \ol{\mathcal{P}_1} \subset \mathcal{P}_p$ and $\lambda(z) \in  \mathcal{P}_p^\psi$,
we can derive a contradiction.
\end{proof}

Next, we prove that
\begin{lemma}
\label{C0-C1b}
Suppose $u \in C^3 (M)$ is an admissible solution to the equation \eqref{eqn} and \eqref{eqn-b}.
We have the following estimates
\[
\sup_M |u| \leq C \; \mbox{and} \; \sup_{\partial M} |\nabla u | \leq C,
\]
where $C$ depends on $(M, \omega)$, $\chi$ and $\ul u$.
\end{lemma}

\begin{proof}
Define $\tilde \psi = \psi^{1/C_n^p}$.
By the inequality of arithmetic and geometric means, we have
\[
\frac{ p  \omega^{j\bar k} (\chi_{\overline{k} j} + u_{ \overline{k} j} ) }{n} \geq f (\lambda) \geq \inf \tilde \psi > 0.
\]
Let $\ol u$ be the solution to the following equation
\[
\omega^{j\bar k} (\chi_{\bar k j} + \partial_j \partial_{\bar k} \ol u) = 0 \; \mbox{in} \; M
\;\; \mbox{and} \;\; \ol u = \varphi \; \mbox{on} \; \partial M.
\]
Then, by the maximum principle we can derive that $\ul u \leq u \leq \ol u$ on $M$.
Since $\ul u = u = \ol u$ on $\partial M$, we get the desired gradient estimate on $\partial M$.
\end{proof}

We now recall second order estimates for admissible solutions from \cite{Gabor} and \cite{H-M-W}.

\begin{proposition}
Suppose $u \in C^4 (M)$ is an admissible solution to the equation \eqref{eqn} and \eqref{eqn-b}.
Then, we have
\begin{equation}
\label{C2-global}
\sup_M |\sqrt{-1}\partial \bar \partial u | \leq C (K + \sup_{\partial M} |\sqrt{-1} \partial \bar \partial u|),
\end{equation}
where $C$ depends on $(M, \omega)$, $|\tilde \psi|_{C^2}$, $\ul u$ and $\chi$.
\end{proposition}

\begin{proof}
It is easy to see that our operator satisfies the assumptions in \cite{Gabor}.
Define the following two functions,
\[
\varphi (s) = -\frac{1}{2}\log(1- \frac{s}{2K}) \; \mbox{and} \; \phi (t) = -2 A t + \frac{A \tau }{2} t^2,
\]
where $\tau$ is a small constant to be chosen depending on $|u|_{C^0}$ and $A$ is a large constant to be chosen.
Apply the maximum principle to the following test function:
\[
G=\log \lambda_1 + \varphi (|\nabla u|^{2} ) + \phi (u),
\]
where $\lambda_1: M \rightarrow \mathbb{R}$ is the largest eigenvalue of the
Hermitian endomorphism ${g^i}_j$.
By the argument in \cite{H-M-W} and \cite{Gabor}, we can
conclude the estimate \eqref{C2-global}.
\end{proof}

The main purpose of this paper is to establish the following second order estimate for admissible solutions on the boundary $\partial M$,
which will be completed in the next two sections.

\begin{proposition}
\label{boundary}
Suppose $u \in C^4 (M)$ is an admissible solution to the equation \eqref{eqn} and \eqref{eqn-b} with $p = n-1$.
Then, we have
\begin{equation}
\label{C2}
\sup_{\partial M} |\sqrt{-1}\partial \bar \partial u | \leq C K,
\end{equation}
where $C$ depends on $(M, \omega)$, $\underline{u}$, $\varphi$, $|\tilde \psi|_{C^2}$, $\inf \psi$ and $\chi$.
\end{proposition}

We remark that the assumption $p=n-1$ is only used in deriving the double normal estimate.

\section{Boundary Mixed Normal-Tangential Estimates}

Following the setup in \cite{CP}, let $\Omega$ be a boundary chart of $p\in \partial M$.
Choose coordinates $z = (z^1, \cdots, z^n)$ such that $p$ corresponds to the origin and $\alpha_{\bar k j}(0) = \delta_{k j}$,
where $z^i = x^i + \sqrt{-1} y^i$.
We may rotate the coordinates such that $x^n$ is in the direction of the inner normal vector to $\Omega$ at the origin.
Let $\rho$ be the defining function of the boundary $\partial M$, that is
\[
\partial M \cap \Omega = \{\rho = 0\}, \;\; \Omega \subset \{\rho \leq 0\}, \;\; d\rho \neq 0 \; \mbox{on} \; \partial M.
\]
Without loss of generality, we can assume that
\[
\rho = -x^n + O(|z|^2).
\]
Denote
\[
t^\alpha = y^\alpha, \alpha \in \{1, \cdots, n\}, t^{n+\alpha} = x^\alpha, \alpha \in \{1, \cdots, n-1\}.
\]
We can find a function $\zeta(t) $ such that $\rho (t, \zeta (t)) = 0$ on $\partial M$.
By the boundary condition $u = \ul u$ on $\partial M$, we get the estimate
\begin{equation}
\label{pure-tangential-C2}
|\partial_{t^\alpha} \partial_{t^\beta} u (0)| \leq C, \;\mbox{for}\; 1\leq \alpha, \beta \leq 2n-1,
\end{equation}
where $C$ depends on $(M, \omega)$, $\chi$ and $\ul u$ by the gradient estimate in Lemma \ref{C0-C1b},
as
\begin{equation}
\label{boundary-derivatives}
\partial_{t^\alpha} \partial_{t^\beta} (u - \ul u) (0) = - \partial_{x^n} (u - \ul u) (0) \rho_{t^\alpha t^\beta} (0).
\end{equation}

The goal of this section is to prove
\begin{proposition}
\label{TN}
Suppose $u \in C^4 (\ol M)$ is an admissible solution to the equation \eqref{eqn} and \eqref{eqn-b}.
Then, we have
\begin{equation}
\label{tangential-normal-C2}
|g_{\bar n i}| \leq C K^{1/2} \; \mbox{for}\; i = 1, \cdots, n-1,
\end{equation}
where $C$ depends on $(M, \omega)$, $\ul u$, $\varphi$, $|\tilde \psi|_{C^2}$, $\inf_M \psi$ and $\chi$.
\end{proposition}

Before proving the above result, we need a few lemmas.
First is a Lemma due to B. Guan. Note that $f$ satisfies the assumptions as in \cite{G}.
Since $\sum f_i \lambda_i \geq 0$, from Lemma 2.20 and Corollary 2.21 in Guan \cite{G}, we have
\begin{lemma}[\cite{G}]
\label{Guan}
For any $\lambda \in \mathcal{P}_p$, index $r$ and $\epsilon > 0$, we have
\[
\sum f_i |\lambda_i | \leq \epsilon \sum_{i\neq r} f_i \lambda_i^2 + \frac{C}{\epsilon} \sum f_i + C,
\]
where $C$ depends on $n$, $p$ and $f(\lambda)$.
\end{lemma}

\begin{proof}
By (f1)-(f4) and following the proof of Lemma 3.4 in \cite{CP}, one can 
prove the inequality.
\end{proof}

The second lemma is on a barrier function due to B. Guan \cite{Guan98}, which
is used frequently in literatures to establish the normal-tangential estimate on the boundary.
For small $\delta > 0$ to be determined, define
\[
\Omega_\delta = \Omega\cap \{|z| < \delta\}.
\]
Let $d$ denote the distance function to $\partial M$.
Consider the following barrier function
\[
v = (u - \ul u) + \tau d - N d^2,
\]
where $\tau \ll 1 $ and $N \gg 1$ are two positive constant.
Then, we can prove
\begin{lemma}
\label{Fv}
Suppose $u$ is an admissible solution to the equation \eqref{eqn}.
There exist $\tau, N, \varepsilon, \delta > 0$ depending on $(M, \omega)$, $\ul u$, $\sup_M \psi$, $\inf_M \psi$ and $\chi$,
such that $v \geq 0$ and
\[
\tilde F^{p\bar q} \partial_{p}\partial_{\bar q} v \leq - \varepsilon (1 + \mathcal{F}) \; \mbox{in} \; \Omega_\delta.
\]
\end{lemma}

\begin{proof}
We compute at a point with coordinates such that $\omega_{\bar k j} = \delta_{kj}$ and ${g^i}_j = \lambda_j {\delta^i}_j$
with the eigenvalues arranged in the descending order $\lambda_1 \geq \cdots \geq \lambda_n$.

Note that $u \geq \ul u$ in $\Omega$, which implies $v \geq d(\tau - N d) > 0$ in $\Omega_\delta$ for sufficiently small $\delta > 0$.
By direct calculations, we see that
\[\begin{aligned}
\tilde F^{p\bar q} \partial_{p}\partial_{\bar q} v
=&\; \tilde F^{p\bar q} ((u-\ul u)_{\bar q p} + (\tau - 2Nd) d_{\bar q p} - 2 N d_{\bar q} d_p).
\end{aligned}\]
Since $ \chi + \sqrt{-1} \partial \bar\partial \ul u \in P_p (M)$, we can find a positive constant $\epsilon_0 $ such that
\[
\chi + \sqrt{-1} \partial \bar\partial \ul u  - \epsilon_0 \omega \in P_p(M).
\]
By the concavity of $\tilde F$, we have
\[\begin{aligned}
& \tilde F(\omega^{i \bar k} (\chi + \sqrt{-1} \partial \bar\partial \ul u  - \epsilon_0\omega )_{\bar k j})
- \tilde F ({g^i}_j)
\leq  \tilde F^{p\bar q} (\ul u_{\bar q p} - \epsilon_0 \omega_{\bar q p} - u_{\bar q p} ).
\end{aligned}\]
Hence, we derive that
\[
\tilde F^{p\bar q} (u_{\bar q p} -\ul u_{\bar q p} + \epsilon_0 \omega_{\bar q p}) \leq C,
\]
where $C$ depends on $\tilde \psi$ and $\ul u$. Therefore, we obtain that
\[
\tilde F^{p\bar q} \partial_{p}\partial_{\bar q} v
\leq C -\epsilon_0 \mathcal{F} - 2N \tilde F^{\bar q p} d_{\bar q} d_p + C (\tau + 2Nd) \mathcal{F}.
\]
Choosing $\tau \ll 1$ and $\delta \ll 1$, we have
\begin{equation}
\label{F-v}
\tilde F^{p\bar q} \partial_{p}\partial_{\bar q} v
\leq C - \frac{3 \epsilon_0}{4} \mathcal{F} - \frac{N}{2} \tilde F^{\bar 1 1},
\end{equation}
since $|\partial d| = \frac{1}{2}$.

Note that $\tilde F^{\bar i i} = f_i$ and $f_1 \leq \cdots \leq f_n$. Then, we have
\[
N \tilde F^{\bar 1 1} + \epsilon_0 \sum_i \tilde F^{\bar i i}
= (N + \epsilon_0) f_1 + \epsilon_0 f_2 + \cdots + \epsilon_0 f_n.
\]
Since
\[
\begin{aligned}
f_k = \sum_{k\in\{i_1, \cdots, i_p\} } \frac{\frac{1}{C_n^p} f(\lambda) }{\lambda_{i_1} + \cdots + \lambda_{i_p} },
\end{aligned}
\]
it follows that
\[\begin{aligned}
  N \tilde F^{\bar 1 1} + \epsilon_0 \sum_i \tilde F^{\bar i i}
\geq &\; N \sum_{1\notin \{i_1, \cdots, i_{p-1}\}} \frac{ \frac{1}{C_n^p} f(\lambda)  }{\lambda_1 + \lambda_{i_1} + \cdots + \lambda_{i_{p-1}}} \\
&\; + \epsilon_0 \sum_{1, 2\notin \{i_1, \cdots, i_{p-1}\}} \frac{ \frac{1}{C_n^p} f(\lambda)  }{\lambda_2 + \lambda_{i_1} + \cdots + \lambda_{i_{p-1}}}\\
&\; + \cdots + \epsilon_0 \; \frac{ \frac{1}{C_n^p} f(\lambda)  }{\lambda_{n-p+1}  + \cdots + \lambda_n  }.
\end{aligned}\]
By the inequality of arithmetic and geometric means, we obtain
\[\begin{aligned}
 N \tilde F^{\bar 1 1} + \epsilon_0 \sum_i \tilde F^{\bar i i}
 \geq N^{p/n} \epsilon_0^{(n-p)/n}.
\end{aligned}\]
Choosing $N$ sufficiently large, we derive from \eqref{F-v} that
\begin{equation}
\label{F-v'}
\tilde F^{p\bar q} \partial_{p}\partial_{\bar q} v
\leq - \frac{ \epsilon_0}{4} \mathcal{F}
\leq - \frac{ \epsilon_0}{8} (\mathcal{F} + 1)
\end{equation}
as $\mathcal{F} \geq p \geq 1$. The lemma is proved.

\end{proof}

In $\Omega_\delta$, we define the tangential vector fields to the level sets of $\rho$:
\[
T_\alpha = \frac{\partial }{\partial t^\alpha} - \frac{\rho_{t^\alpha}}{\rho_{x^n}} \frac{\partial}{\partial x^n},\;
\mbox{where}\; \alpha = 1, \cdots, 2n-1.
\]
We introduce the notation $\mathcal{E}$ to denote terms which can be estimated by
\[
|\mathcal{E}| \leq C (1 + K^{1/2}) \mathcal{F} + C \sum f_i |\lambda_i| + C,
\]
where $C$ only depends on $(M, \omega)$, $\chi$, and $\ul u$. As in Collins-Picard \cite{CP},
we have
\begin{lemma}
\label{FTu}
There exists $\delta > 0$ depending on $(M, \omega)$ such that we can estimate
\begin{equation}
\label{FTu-ineq}
|\tilde F^{p\bar q} \partial_p \partial_{\bar q} T_\alpha (u- \ul u)|
\leq \frac{1}{K^{1/2}} \tilde F^{p \bar q} \partial_p \partial_{y^n} (u- \ul u) \partial_{\bar q} \partial_{y^n} (u- \ul u) + \mathcal{E}.
\end{equation}
\end{lemma}

\begin{proof}
The proof of the above lemma is the same as that of Lemma 4.3 in \cite{CP}.
So we only sketch it here.
Note that
\[\begin{aligned}
 \tilde F^{p\bar q} \partial_p \partial_{\bar q} T_\alpha (u- \ul u)
= &\; \tilde F^{p\bar q} \partial_p \partial_{\bar q} \partial_{t^\alpha} (u - \ul u) -
\frac{\rho_{t^\alpha}}{\rho_{x^n}} \tilde F^{p \bar q} \partial_p \partial_{\bar q} \partial_{x^n} (u -\ul u)\\
- &\; 2 \text{Re} \Big( \tilde F^{p\bar q} \Big( \partial_p \frac{\rho_{t^\alpha}}{\rho_{x^n}}\Big) \partial_{\bar q} \partial_{x^n} (u - \ul u)\Big)
 - \tilde F^{p\bar q} \Big( \partial_p \partial_{\bar q} \frac{\rho_{t^\alpha}}{\rho_{x^n}} \Big)  \partial_{x^n} (u- \ul u).
\end{aligned}\]
The third order terms and the last term on the right hand side can be controlled by $|\mathcal{E}|$.
Since
\[
 \partial_{\bar q} \partial_{x^n} (u - \ul u)
= 2  \partial_{\bar q} \partial_{n} (u - \ul u) + \sqrt{-1} \partial_{\bar q} \partial_{y^n} (u - \ul u)
\]
and
\[
2 \tilde F^{p\bar q} \Big( \partial_p \frac{\rho_{t^\alpha}}{\rho_{x^n}} \Big) \partial_{\bar q} \partial_{n} (u- \ul u)
\leq C \mathcal{F} + 2 \tilde F^{p\bar q} \partial_p \frac{\rho_{t^\alpha}}{\rho_{x^n}} g_{\bar q n}
\leq C \mathcal{F} + C \sum f_i |\lambda_i|,
\]
the third term on the right hand side can be estimated as
\[\begin{aligned}
& \tilde F^{p\bar q} \Big( \partial_p \frac{\rho_{t^\alpha}}{\rho_{x^n}}\Big) \partial_{\bar q} \partial_{x^n} (u - \ul u)\\
\leq &\; \frac{1}{K^{1/2}} \tilde F^{p\bar q} \partial_{p} \partial_{y^n} (u - \ul u)\partial_{\bar q} \partial_{y^n} (u - \ul u)
+ C \sum f_i |\lambda_i| + CK^{1/2} \mathcal{F}
\end{aligned}\]
by Cauchy-Schwarz inequality.
Therefore, the lemma follows.

\end{proof}

To kill the first term in \eqref{FTu-ineq}, Collins-Picard \cite{CP} used the quadratic gradient term 
in constructing the barrier function.
\begin{lemma}
\label{F-partial-u}
The quadratic gradient term $\frac{1}{K^{1/2}} (\partial_{y^i} (u - \ul u) )^2$ can be estimated as
\[
\frac{1}{K^{1/2}} \tilde F^{p\bar q}  \partial_p \partial_{\bar q} (\partial_{y^i} (u - \ul u) )^2
\geq \frac{2}{K^{1/2}} \tilde F^{p\bar q}  \partial_{p}\partial_{y^i} (u - \ul u) \partial_{\bar q}\partial_{y^i} (u - \ul u)
+ \mathcal{E}.
\]
\end{lemma}

\begin{proof}
We have
\[\begin{aligned}
& \frac{1}{K^{1/2}} \tilde F^{p\bar q}  \partial_p \partial_{\bar q} (\partial_{y^i} (u - \ul u) )^2\\
=&\; \frac{2}{K^{1/2}} \tilde F^{p\bar q}  \partial_p  \partial_{y^i}(u - \ul u) \partial_{\bar q} \partial_{y^i} (u - \ul u)
+ \frac{2}{K^{1/2}} \partial_{y^i} (u - \ul u)  \tilde F^{p\bar q}  \partial_p \partial_{\bar q} \partial_{y^i} (u - \ul u) \\
\geq &\; \frac{2}{K^{1/2}} \tilde F^{p\bar q}  \partial_p  \partial_{y^i}(u - \ul u) \partial_{\bar q} \partial_{y^i} (u - \ul u)
- C\mathcal{F} - C,
\end{aligned}\]
where we used the first derivatives of the equation $\tilde F ({g^i}_j) = \tilde \psi$ in the inequality.
\end{proof}

Let $e_a = {e^i}_a \partial_i$ be a local orthonormal frame of $T^{1,0} X$ such that $\{e_a\}_{a=1}^{n-1}$
are tangential to the level sets of $\rho$.
The following lemma can be obtained following the lines of subsection 4.5 in \cite{CP}.
\begin{lemma}
\label{F-gradient-u}
There exists an index $1\leq r \leq n$ such that the following estimate hold:
\[\begin{aligned}
&\; \frac{1}{K^{1/2}} \tilde F^{p\bar q}  \partial_p \partial_{\bar q} \sum_{a=1}^{n-1} |\nabla_a (u- \ul u)|^2\\
\geq &\; \frac{1}{2n K^{1/2}} \sum_{i\neq r} f_i \lambda_i^2 -
\frac{1}{K^{1/2}} \sum_{i=1}^n \tilde F^{p\bar q}  \partial_{p}\partial_{y^i} (u - \ul u) \partial_{\bar q}\partial_{y^i} (u - \ul u)
+ \mathcal{E}.
\end{aligned}\]
\end{lemma}

\begin{proof}
Following the calculations in \cite{CP}, one can arrive at
\[\begin{aligned}
&\; \frac{1}{K^{1/2}} \tilde F^{p\bar q}  \partial_p \partial_{\bar q} \sum_{a=1}^{n-1} |\nabla_a (u- \ul u)|^2\\
\geq &\; \frac{1}{2 K^{1/2}} \sum_{a=1}^{n-1} \tilde F^{p\bar q} g_{\bar q a} g_{\bar a p} -
\frac{1}{K^{1/2}} \sum_{i=1}^n \tilde F^{p\bar q}  \partial_{p}\partial_{y^i} (u - \ul u) \partial_{\bar q}\partial_{y^i} (u - \ul u)
+ \mathcal{E}
\end{aligned}\]
and 
\[
\sum_{a=1}^{n-1} \tilde F^{p\bar q} g_{\bar q a} g_{\bar a p} \geq \frac{1}{n} \sum_{i\neq r} f_i \lambda_i^2
\]
for some index $r$. Then, the lemma is proved.
\end{proof}

Now we are in position to prove Proposition \ref{TN}.
\begin{proof}[Proof of Proposition \ref{TN}]
Consider the following barrier function due to Collins-Picard,
\[
\Psi = A K^{1/2} v + B K^{1/2} |z|^2 - \frac{1}{K^{1/2}} \sum_{i=1}^{n} (\partial_{y^i} (u - \ul u) )^2
-  \frac{1}{K^{1/2}} \sum_{a=1}^{n-1} |\nabla_a (u- \ul u)|^2.
\]
Denote
$\mathcal{L} = \tilde F^{p\bar q}  \partial_p \partial_{\bar q}$.
By Lemma \ref{Fv}, Lemma \ref{FTu}, Lemma \ref{F-partial-u} and Lemma \ref{F-gradient-u}, we obtain
\[\begin{aligned}
& \mathcal{L} (\Psi + T_{\alpha} (u - \ul u))\\
\leq &\; - A \varepsilon K^{1/2} (1 + \mathcal{F}) + B K^{1/2} \sum_i f_i - \frac{1}{2n K^{1/2}} \sum_{i\neq r} f_i \lambda_i^2\\
&\; - \frac{1}{K^{1/2}} \sum_{i=1}^{n-1} \tilde F^{p\bar q}  \partial_{p}\partial_{y^i} (u - \ul u) \partial_{\bar q}\partial_{y^i} (u - \ul u)
+ |\mathcal{E}|.
\end{aligned}\]
Choosing $A \gg 1$, we get
\[\begin{aligned}
& \mathcal{L} (\Psi + T_{\alpha} (u - \ul u))\\
\leq &\; - \frac{A \varepsilon}{2} K^{1/2} (1 + \mathcal{F}) - \frac{1}{2n K^{1/2}} \sum_{i\neq r} f_i \lambda_i^2 + C \sum_i f_i |\lambda_i|.
\end{aligned}\]
By Lemma \ref{Guan} with $\epsilon = \frac{1}{2nCK^{1/2}}$, we see that
\[
\mathcal{L} (\Psi + T_{\alpha} (u - \ul u)) \leq - \frac{A \varepsilon}{2} K^{1/2} (1 + \mathcal{F})
+ \frac{C}{\epsilon} \sum_i f_i + C \leq 0
\]
for $A \gg 1$.

On the boundary $\partial M \cap \Omega_\delta$,
we have $\partial_{y^i} (u - \ul u) = - \partial_{x^n} (u - \ul u) \partial_{y^i} \zeta$
and by Lemma \ref{C0-C1b} we obtain
\[
(\partial_{y^i} (u - \ul u) )^2 \leq C |z|^2.
\]
On $\partial M$, since $\zeta_{t^\alpha} = -\frac{\rho_{t^\alpha}}{\rho_{x^n}}$, we have
\[
T_{\alpha } (u - \ul u) = \partial_{t^\alpha} (u - \ul u) + \zeta_{t^\alpha} \partial_{x^n} (u- \ul u) = 0
\]
and $\nabla_a (u- \ul u) =0$. Recall that $v\geq 0$. Hence, we see that
\[
\Psi + T_\alpha (u - \ul u) \geq B K^{1/2} |z|^2 - \frac{C}{K^{1/2}} |z|^2 \geq 0
\]
on $\partial M \cap \Omega_\delta$ for $B \gg 1$.
On the piece $\partial B_\delta \cap \Omega_\delta$, for sufficiently large $B$, we have
\[
\Psi + T_\alpha (u - \ul u) \geq B K^{1/2} \delta^2 - C K^{1/2} \geq 0.
\]
It follows that $\Psi + T_\alpha (u - \ul u) \geq 0$ on $\Omega_\delta$ by the maximum principle.
Since
\[
[\Psi + T_\alpha (u - \ul u)](0) = 0,
\]
we therefore derive that
\[
AK^{1/2} \partial_{x^n} v (0) + \partial_{x^n} \partial_{t^\alpha} (u - \ul u) (0)
- \partial_{x^n} (\frac{\rho_{t^\alpha}}{\rho_{x^n}})(0) \partial_{x^n} (u - \ul u)(0) \geq 0.
\]
By the boundary gradient estimate Lemma \ref{C0-C1b}, we conclude
\[
\partial_{x^n} \partial_{t^\alpha} (u - \ul u) (0) \geq - CK^{1/2}.
\]
Similarly, consider $\Psi - T_{\alpha} (u - \ul u)$ to conclude that
\[
\partial_{x^n} \partial_{t^\alpha} (u - \ul u) (0) \leq CK^{1/2}.
\]
It follows that
$|g_{\bar n i}| (0) \leq CK^{1/2}$,
which finishes the proof of Proposition \ref{TN}.
\end{proof}

\section{Boundary Double Normal Estimate}

Let $p\in \partial M$ be a boundary point. As before, we choose coordinates $z= (z^1, \cdots, z^n)$
such that $p$ corresponds to the origin and $\omega_{\bar k j} (0) = \delta_{kj}$.
By rotating the coordinates we can assume that $\frac{\partial}{\partial x^n}$ is the inner normal vector to the
boudary $\partial \Omega$ at $p$.
Furthermore, the matrix $g_{\bar k j} = \chi_{\bar k j} + u_{\bar k j}$ has the form
\[
 g =
\left( \begin{array}{ccccc}
 \lambda_1' & 0  & \cdots & 0 & g_{\bar 1 n} \\
 0 & \lambda_2'  & \cdots & 0 & g_{\bar 2 n}\\
\vdots & \vdots &  \ddots & \vdots & \vdots \\
 0 & 0 & \cdots & \lambda_{n-1}' & g_{\ol{n-1} n}\\
 g_{\bar n 1} & g_{\bar n 2} & \cdots & g_{\bar n n-1} & g_{\bar n n}
\end{array} \right).
\]
Since $g \in P_p \subset P_n$, we have $\tr_\omega g > 0$. Therefore, we only need to
derive an upper bound for $g_{\bar n n}$ as $| \lambda_i' | \leq C$ for $1 \leq i \leq n-1$ by
Lemma \ref{C0-C1b} and the estimate \eqref{pure-tangential-C2}.
Note that for $p = n-1$, in local coordinates, the equation \eqref{eqn} can be written in an equivalent form as below
\begin{equation}
\label{eqn'}
\begin{aligned}
&\; (\tr_{\omega} g - g_{\bar 1 1}) \cdots (\tr_{\omega} g - g_{\bar n n}) \\
=&\; \psi + \sum_{i=1}^{n-1} |g_{\bar n i}|^2 (\tr_{\omega} g - g_{\bar 1 1}) \cdots
\widehat{(\tr_{\omega} g - g_{\bar i i})} \cdots (\tr_{\omega} g - g_{\ol{n-1} n-1}),
\end{aligned}
\end{equation}
where the wide hat is used to denote the term which does not appear.
Without loss of generality, we also assume $\lambda_1' \geq \cdots \geq \lambda_{n-1}'$.
By \eqref{eqn'} and the tangential-normal estimate \eqref{tangential-normal-C2}, we can estimate
\[
[ (\tr_{\omega} g - g_{\bar 1 1}) (\tr_{\omega} g - g_{\bar n n} ) -CK ]
(\tr_{\omega} g - g_{\bar 2 2}) \cdots (\tr_{\omega} g - g_{\ol{n-1} n-1} ) \leq C.
\]

Now we divide the proof into two cases. First, if $(\tr_{\omega} g - g_{\bar 1 1}) (\tr_{\omega} g - g_{\bar n n} ) - CK > 0$,
we have, by Lemma \ref{thm-lowerbound} below,
\[
c_0^{n-2} [(\tr_{\omega} g - g_{\bar 1 1}) (\tr_{\omega} g - g_{\bar n n} ) - CK] \leq C,
\]
from which
we derive that
\begin{equation}
\label{normal-normal-C2}
\tr_{\omega} g - g_{\bar 1 1} \leq CK.
\end{equation}
Second, if $(\tr_{\omega} g - g_{\bar 1 1}) (\tr_{\omega} g - g_{\bar n n} ) - CK \leq 0$,
from this inequality we can directly get by Lemma \ref{thm-lowerbound}
that
\begin{equation}
\label{doubel-normal-C2}
\tr_{\omega} g - g_{\bar 1 1} \leq C K,
\end{equation}
which finishes the proof of Proposition \ref{boundary}.

Therefore,
our main goal in this section is to prove the following.
\begin{lemma}
\label{thm-lowerbound}
Let $\lambda' \in \mathbb{R}^{n-1}$ be the eigenvalues of the endomorphism $g$ restricted to the subbundle $T^{1,0} \partial M$.
Then, we have
\begin{equation}
\label{lowerbound}
\lambda_1' + \cdots + \lambda_{n-1}' \geq c_0
\end{equation}
for some uniform positive constant $c_0$ depending on $(M, \omega)$, $\chi$, $\psi$ and $\ul u$.
\end{lemma}

\begin{proof}[Proof of Theorem \ref{thm-lowerbound}]
For any given $p \in \partial M$, choose coordinates such that the point $p$ corresponds to the origin
and the metric $\omega_{\bar k j} (0) = \delta_{kj}$.
By orthogonally rotating the coordinates, we can assume that $T_p^{1,0}(\partial M)$
is spanned by $\{\frac{\partial}{\partial z^1}, \cdots, \frac{\partial}{\partial z^{n-1}}\}$
and $x^n$ is in the direction of the inner normal vector at the origin.

We shall use Greek indices $\alpha,\beta \in \{1, \cdots, n-1\}$ for tangential directions below.
Also, we will denote the eigenvalues of
$(\chi_{\bar \alpha \beta} + u_{\bar \alpha \beta} ) (0)$ and
$(\chi_{\bar \alpha \beta} + \ul u_{\bar \alpha \beta} ) (0)$
by $\lambda' = (\lambda_1', \cdots, \lambda_{n-1}')$ and
$\ul \lambda' = (\ul{\lambda_1'}, \cdots, \ul{\lambda_{n-1}'})$, respectively.
From the formula \eqref{boundary-derivatives}, we have
\begin{equation}
\label{boundary-derivatives-2}
(\chi_{\bar \alpha \beta} + u_{\bar \alpha \beta} ) (0) = (\chi_{\bar \alpha \beta} + \ul u_{\bar \alpha \beta} ) (0)
- (u - \ul u)_{x^n} (0) \rho_{\bar \alpha \beta} (0).
\end{equation}
For convenience, we denote
\[
\eta_0 := (u - \ul u)_{x^n} (0).
\]
Note that $\eta_0 \geq 0$.

Suppose first that $\eta_0 = 0$. Then, we have $\lambda' = \ul \lambda'$.
Since $\ul u$ is admissible, it means $(\chi_{\bar k j} + \ul u_{\bar k j} ) (0) \in P_{n-1}$.
There exists a positive constant $\varrho_0$ depending on $\ul u$ such that
$\ul{\lambda_1'} + \cdots + \ul{\lambda_{n-1}'} \geq \varrho_0$,
which finishes the proof of \eqref{lowerbound}.

Now, we may assume $0< \eta_0 \leq C$.
For a real parameter $t$, consider the family of $(n-1)\times (n-1)$ matrices
\[
A_t = t (\chi_{\bar \alpha \beta} + \ul u_{\bar \alpha \beta} ) (0) - (u - \ul u)_{x^n} (0) \rho_{\bar \alpha \beta} (0).
\]
Note that, at $t = 1$, we have $\lambda (A_1) = \lambda' \in \mathcal{P}_{n-1} \subset \mathbb{R}^{n-1}$.
When $t$ goes to $-\infty$, we have $\lambda (A_t) \notin \mathcal{P}_{n-1}$.
Let $t_0 < 1$ be the first value of $t$ that when decreasing from $+\infty$ the eigenvalues of $A_t$
hit the boundary of the cone:
$\lambda(A_{t_0}) \in \partial \mathcal{P}_{n-1}$.
Our goal is to show that $t_0$ cannot be too close to $1$, i.e.
\begin{equation}
\label{t-bound}
t_0 \leq 1 - \kappa_0,
\end{equation}
for a uniform $\kappa_0 > 0$.
Observe that
\[
(1- t_0) (\chi_{\bar \alpha \beta} + \ul u_{\bar \alpha \beta} ) (0) + A_{t_0} =
(\chi_{\bar \alpha \beta} + u_{\bar \alpha \beta} ) (0).
\]
Assuming \eqref{t-bound}, we can prove \eqref{lowerbound}.
We have
\[\begin{aligned}
\lambda_1' + \cdots + \lambda_{n-1}' = (1-t_0) ( \ul{\lambda_1'} + \cdots + \ul{\lambda_{n-1}'} )
\geq \kappa_0 \varrho_0.
\end{aligned}\]
So \eqref{lowerbound} is established.

Now we prove \eqref{t-bound}.
Let $\Omega$ be a boundary chart containing $p \in \partial M$. Define $\Omega_\delta = \Omega \cap B_{\delta} (0)$ for small $\delta > 0$.
To prove \eqref{t-bound}, 
we consider the following auxiliary functions of Caffarelli-Nirenberg-Spruck \cite{CNS3},
which are defined in $\Omega_\delta$:
\[\begin{aligned}
D(z) =&\; - \rho(z) + \tau |z|^2 \geq 0,\\
\Phi(z) =&\; \ul u(z) - \frac{\eta_0}{t_0} \rho(z) + (l_i z^i + l_{\bar i} \bar z^{i}) \rho (z) + L D(z)^2\\
\Psi (z) = &\; \Phi (z) + \varepsilon (|z|^2 - \frac{1}{C_0} x^n).
\end{aligned}\]
The parameters will be chosen carefully such that $\tau, \varepsilon > 0$ are small constants, $L, C_0 > 1$ are large constants,
and $l_i \in \mathbb{C}$ are bounded with $l_{\bar i} = \ol{l_i}$.
By Lemma \ref{barrier} below,
we have $u \leq \Psi$ on $\Omega_\delta$. Since $\Psi (0) = \ul u(0) = u(0)$, we have
$\partial_{x^n} \Psi (0) \geq \partial_{x^n} u (0)$, which implies that
\[
- (u - \ul u)_{x^n} (0) \geq \frac{\varepsilon}{C_0} + \frac{\eta_0}{t_0} \partial_{x^n} \rho(0).
\]
This is equivalent to
\[
t_0 \leq \frac{1}{1 + \varepsilon \eta_0^{-1} C_0^{-1}}
\]
as $ \partial_{x^n} \rho (0) = -1$. The proof of \eqref{t-bound} is completed with
$\kappa_0 = \frac{\varepsilon C_0^{-1}}{\eta_0 + \varepsilon C_0^{-1}} > 0$ and
the proof of Lemma \ref{thm-lowerbound} is finished.
\end{proof}

Now we show
\begin{lemma}
\label{barrier}
Suppose $t_0 \geq 1/2$. There exist parameters $\delta, \tau, \varepsilon, L, C_0, l_i$ depending only on $(M, \omega)$, $\chi$,
$\inf_M \psi$ and $\ul u$ such that
\[
u (z) \leq \Psi (z)\; \mbox{on}\; \Omega_\delta.
\]
\end{lemma}

\begin{proof}[Proof of Lemma \ref{barrier}]

Note that the normal vector to $ \partial \mathcal{P}_{n-1} \subset \mathbb{R}^{n-1}$ is $(1, \cdots, 1) \in \mathbb{R}^{n-1}$.
Let
\[
\xi_a = \sum_{i=1}^{n} {\xi^i}_a \frac{\partial}{\partial z^i}
\]
be a local orthonormal frame of $T^{1,0} M$ defined in $\Omega_\delta$
such that $\{\xi_a\}_{a= 1}^{n-1}$ span the holomorphic tangent space of the level sets of $D(z)$.
Without loss of generality, we can assume $\xi_a (0) = \frac{\partial}{\partial z^a}$ for $a\in \{1, \cdots, n-1\}$.
Now we define a local operator on $(1,1)$ forms as follows, for $\beta = \sqrt{-1} \beta_{\bar k j} dz^j \wedge d \bar{z}^k$,
\[
\Lambda \beta = \frac{1}{\sqrt{-1}} \sum_{a=1}^{n-1} \beta(\xi_a, \ol{\xi_a})
= \sum_{a=1}^{n-1} {\xi^j}_a \ol{{\xi^k}_a} \beta_{\bar k j}.
\]
By Lemma \ref{L-Phi} below, we have, in $\Omega_\delta$,
\[
 \Lambda (\chi + \sqrt{-1} \partial \bar \partial \Phi) \leq 0.
\]

Assuming the above, we are in position to establish $u(z) \leq \Phi (z)$ in $\Omega_\delta$.
Let $W= \omega^{-1} (\chi + \sqrt{-1} \partial \bar \partial \Phi)$
with eigenvalues $\mu_1 \geq \cdots \geq \mu_n$.
At a point $z \in \Omega_\delta$, take new coordinates such that $\omega_{\bar k j}= \delta_{kj}$
and ${W^i}_j = \mu_i {\delta^i}_j$.
In these new coordinates, we can write $\xi_a = {\xi^i}_a \partial_{z^i}$ with
${\xi^i}_a$ a unitary matrix and
\[
\Lambda (\chi + \sqrt{-1} \partial \bar \partial \Phi)(z) = \sum_{a=1}^{n-1} \sum_{i=1}^n |{\xi^i}_a|^2 \mu_i.
\]
Let $\xi_0$ be such that $\{\xi_a\}_{a=0}^{n-1}$ is a local unitary frame for $T^{1,0}X$.
Then, by Lemma \ref{L-Phi},
\[\begin{aligned}
0\geq \Lambda (\chi + \sqrt{-1} \partial \bar \partial \Phi)
= \sum_{i=1}^n (1- |{\xi^i}_0|^2) \mu_i
\geq \sum_{i=2}^{n} \mu_i.
\end{aligned}\]
which means the vector $(\mu_2, \cdots, \mu_n)$ is outside of $\mathcal{P}_{n-1} \subset \mathbb{R}^{n-1}$.
It follows that $\mu = (\mu_1, \cdots, \mu_n)$ is outside of $\mathcal{P}_{n-1} \subset \mathbb{R}^{n}$.
Let $\sigma = \inf_M \psi > 0$ and
\[
\mathcal{P}_{n-1}^\sigma = \{\lambda \in \mathcal{P}_{n-1} : f (\lambda) \geq \sigma\}.
\]
It is easy to see that
\[
\lambda( \chi_{\bar k j} + \Psi_{\bar k j}) = \lambda(\chi_{\bar k j} + \Phi_{\bar k j} + \varepsilon \delta_{kj}) \notin \mathcal{P}_{n-1}^\sigma
\]
for sufficiently small $\varepsilon > 0$.
Observe that
\[
\Phi - u \geq -C + L\tau^2 \delta^4\;\mbox{on}\; \partial B_{\delta} \cap M;\;
\Phi - u = L \tau^2 |z|^4\; \mbox{on} \; B_{\delta} \cap \partial M.
\]
So, for sufficiently large $L$, $\Phi - u \geq 0$ on $\partial (B_{\delta} \cap M)$.

Now it is easy to see, on $\partial B_{\delta} \cap M$,
\[
\Psi - u \geq \varepsilon \delta^2 - \frac{\varepsilon}{C_0} x^n \geq 0,
\]
and, on $B_{\delta} \cap \partial M$,
\[
\Psi- u \geq \varepsilon |z|^2 - \frac{\varepsilon}{C_0} x^n \geq  \varepsilon |z|^2 - \frac{\varepsilon}{C_0} O(|z|^2) \geq 0,
\]
for sufficiently large $C_0$. By the maximum principle Lemma \ref{MP}, we conclude that
$u (z) \leq \Psi (z)$ in $\Omega_\delta$, which finishes the proof of Lemma \ref{barrier}.
\end{proof}

Finally, we estimate $\Phi$.
\begin{lemma}
\label{L-Phi}
Let $1/2 \leq t_0 \leq 1$. There exists parameters $\tau$, $L$, $l_i$, $\delta$ depending on $(M, \omega)$, $\chi$, $\inf_M \psi$
and $\ul u$, such that
\[
\Lambda (\chi + \sqrt{-1} \partial \bar \partial \Phi) \leq 0, \; \mbox{in} \; \Omega_\delta.
\]
\end{lemma}

\begin{proof}[Proof of Lemma \ref{L-Phi}]
Let Greek indices $\alpha, \beta$ take values in $1, \cdots, n-1$.
By the definition of $A_{t_0}$, we see that
\[
0 = t_0 \sum_\alpha
 (\chi_{\bar \alpha \alpha} + {\ul{u}}_{\bar \alpha \alpha} )(0)
- \eta_0 \sum_\alpha \rho_{\bar \alpha \alpha} (0).
\]
Currently, $\ul{\lambda'}$ corresponds to the eigenvalues of $(\chi_{\bar \alpha \beta} + {\ul{u}}_{\bar \alpha \beta} )(0)$.
We can extends $\ul{\lambda'}$ to other points as in Collins-Picard \cite{CP}.
Since $\ul u$ is an admissible subsolution, the vector $\ul{\lambda'}$ lives in a compact set of $\mathcal{P}_{n-1} \subset \mathbb{R}^{n-1}$.
Therefore,
\[
\sum_\alpha \rho_{\bar \alpha \alpha} (0) \geq \frac{t_0}{\eta_0} \sum_\alpha \ul{\lambda'_\alpha} (0)
\geq \frac{1}{2 \eta_0} \inf_{p\in \partial M} \sum_\alpha \ul{\lambda'_\alpha} (p)
\geq \frac{\theta}{2 \eta_0},
\]
where $\theta > 0$ is a constant only depending on $(M, \omega)$, $\chi$, $\ul u$. Since $|\nabla u| \leq C$ on $\partial M$
as proved in Lemma \ref{C0-C1b}, we know $|\eta_0| = | (u - \ul u)_{x^n} (0)| \leq C$, which implies that
\begin{equation}
\label{lower-rho}
\sum_{\alpha} \rho_{\bar \alpha \alpha} (0) \geq \theta_0
\end{equation}
for $\theta_0 > 0$ is a constant only depending on $(M, \omega)$, $\chi$, $\ul u$.

Now, we compute the the quantity in the lemma.
We have
\[
\Lambda (\chi + \sqrt{-1} \partial \bar \partial \Phi) = T_1 + T_2 + T_3 + T_4,
\]
where
\[\begin{aligned}
T_1 = &\; \sum_{a = 1}^{n-1} {\xi^j}_a  \ol{ {\xi^{k}}_a }
[ (\chi_{\bar k j} + \ul{u}_{\bar k j}) - \frac{\eta_0}{t_0} \rho_{\bar kj}]\\
T_2 = &\; \sum_{a = 1}^{n-1} {\xi^j}_a  \ol{ {\xi^{k}}_a } (l_{\bar k} \rho_j + l_j \rho_{\bar k})\\
T_3 = &\; \sum_{a = 1}^{n-1} {\xi^j}_a  \ol{ {\xi^{k}}_a } (l_i z^i + l_{\bar i} \bar z^i) \rho_{\bar k j}\\
T_4 = &\; 2 L \sum_{a = 1}^{n-1} {\xi^j}_a  \ol{ {\xi^{k}}_a }\partial_j D \partial_{\bar k} D
+ 2LD \sum_{a = 1}^{n-1} {\xi^j}_a  \ol{ {\xi^{k}}_a } \partial_j \partial_{\bar k} D.
\end{aligned}\]
At the origin, we have $T_1 (0) = 0$. Therefore, we see
\[
T_1 = m_i z^i +m_{\bar i} \bar z^{i} + O(|z|^2),
\]
for some bounded constants $m_i$ depending only on $(M, \omega), \chi, \underline{u}$.
Since the vector fields $\xi_a$ are tangential to the level sets of $D(z)$,
direct computation gives that
\[
0 = {\xi^j}_a\partial_j D = -{\xi^j}_a \rho_j + \tau {\xi^j}_a \bar z^j.
\]
Substituting the above equality into $T_2$, we get
\[\begin{aligned}
T_2 = &\; \tau \sum_{a=1}^{n-1} {\xi^j}_a \overline{{\xi^k}_a} (l_{\bar k} \bar z^j + l_j z^k)\\
= &\; \tau \sum_{a=1}^{n-1} \big({\xi^j}_a (0) \overline{{\xi^k}_a} (0) + O(|z|)\big) (l_{\bar k} \bar z^j + l_j z^k)\\
= &\; \tau \sum_{a=1}^{n-1} (l_a z^a + l_{\bar a} \bar z^a) + O(|z|^2),
\end{aligned}\]
where in the third equality we used $\xi_a (0) = \partial_a$.
Similarly, we have
\[\begin{aligned}
T_3 = &\; \sum_{a=1}^{n-1} {\xi^j}_a \overline{{\xi^k}_a} (l_i z^i + l_{\bar i} \bar z^i) \rho_{\bar k j} (0) + O(|z|^2)\\
= &\; \sum_{a=1}^{n-1} (l_i z^i + l_{\bar i} \bar z^i) {\xi^j}_a (0) \overline{{\xi^k}_a}(0) \rho_{\bar k j} (0) + O(|z|^2)\\
= &\; (l_i z^i + l_{\bar i} \bar z^i) \sum_{a=1}^{n-1} \rho_{\bar a a} (0) + O(|z|^2).
\end{aligned}\]
Therefore, we obtain
\[\begin{aligned}
T_1 + T_2 + T_3
= &\; 2 \text{Re} \sum_{i=1}^{n-1} \{m_i + \tau l_i + l_i \sum_{a=1}^{n-1} \rho_{\bar a a} (0)\}z^i\\
&\; + 2 \text{Re} \{m_n + l_n \sum_{a=1}^{n-1} \rho_{\bar a a} (0) \} z^n + O(|z|^2).
\end{aligned}\]
By \eqref{lower-rho}, for any $\tau > 0$, we can choose $l_i$ such that
\[
l_i \Big( \tau + \sum_{a=1}^{n-1} \rho_{\bar a a} (0) \Big) = - m_i \; \mbox{for} \; 1\leq i \leq n-1
\; \mbox{and} \; l_n  \sum_{a=1}^{n-1} \rho_{\bar a a} (0)  = -m_n.
\]
Note that $|l_i| \leq \frac{|m_i|}{\theta_0}$. We then arrive at
\[
T_1 + T_2 + T_3 \leq C |z|^2.
\]
Recall that $\xi_a$ is tangential to level sets of $D$. So,
\[
T_4 = 2LD \Lambda \sqrt{-1} \partial \bar\partial D.
\]
At the origin,
\[
\Lambda \sqrt{-1} \partial \bar\partial D (0) = \sum_{a=1}^{n-1} (-\rho_{\bar a a} (0) + \tau)
\leq - \frac{\theta_0}{2}
\]
as long as we choose $0 < \tau < \frac{\theta_0}{2(n-1)}$.
We can choose $\delta > 0$ small enough to ensure that
\[
\Lambda \sqrt{-1} \partial \bar\partial D \leq - \frac{\theta_0}{4} \; \mbox{in} \; \Omega_\delta,
\]
and, therefore,
\[
T_4 \leq - \frac{\tau \theta_0 L }{2} |z|^2 \; \mbox{in} \; \Omega_\delta.
\]
It follows that
\[
T_1 + T_2 + T_3 + T_4 \leq C|z|^2 - \frac{\tau \theta_0 L }{2} |z|^2 \leq 0
\]
for $L \geq C (\tau \theta_0)^{-1}$,
which completes the proof of Lemma \ref{L-Phi}.
\end{proof}

\end{document}